\documentclass[11pt,reqno]{amsart}

\usepackage[english]{babel}
\usepackage[utf8]{inputenc}
\usepackage[T1]{fontenc}
\usepackage{lmodern} \normalfont 
\DeclareFontShape{T1}{lmr}{bx}{sc} { <-> ssub * cmr/bx/sc }{}
\usepackage{microtype}	

\usepackage{amssymb}
\usepackage{amsmath}
\usepackage{amsthm}
\usepackage{mathtools}
\usepackage{mathrsfs}
\usepackage{accents} 
\usepackage{etoolbox}
\usepackage{siunitx}
\usepackage{dsfont} 
\allowdisplaybreaks 

\usepackage{bm}

\usepackage{tabularx}

\usepackage{graphicx}
\usepackage{color}
\usepackage[dvipsnames]{xcolor}
\usepackage{circuitikz}
\usepackage{tikz}
\usetikzlibrary{calc,positioning,shapes}
\usetikzlibrary{patterns,decorations.pathmorphing,decorations.markings}
\usepackage{pgfplots}
\pgfplotsset{compat=newest}
\usepackage[margin=10pt,font=small,labelfont=bf,labelsep=endash]{caption}
\usepackage{subcaption}
\usepackage{mleftright}

\usepackage{booktabs}
\usepackage{paralist}
\usepackage{soul}
\usepackage{listings}

\numberwithin{equation}{section}

\usepackage{algorithm}
\usepackage{algorithmicx}
\usepackage{algpseudocode}

\usepackage{cite}
\usepackage{multirow} 
\usepackage{enumitem} 
\setlist[enumerate]{label=(\roman*)}
\usepackage{xspace}

\usepackage[colorlinks,linkcolor=teal,citecolor=teal,urlcolor=teal]{hyperref}
\usepackage[nameinlink,noabbrev]{cleveref}

\usepackage{tablefootnote}

\usepackage{lineno}

\theoremstyle{plain}
\newtheorem{theorem}{Theorem}[section]

\newcommand{\dual}[1]{#1^\prime}

\newcommand{\linBoundedOps}{\mathcal{L}}
\newcommand{\linBoundedBilin}{\mathcal{B}}

\newcommand{\intd}{\text{d}}

\newcommand{\field}{\R}

\newcommand{\R}{\ensuremath{\mathbb{R}}}

\newcommand{\norm}[2]{\left\lVert #2 \right\rVert_{#1}}

\newcommand{\rT}[1]{#1^\top}
\newcommand{\rTb}[1]{\rT{\left( #1 \right)}}

\newcommand{\diffFun}[1]{C^{#1}}

\newcommand{\banachAdjoint}[1]{#1^{\vee}}

\newcommand{\evalField}[2][\point]{\mleft.\kern-\nulldelimiterspace{#2}\mright|_{#1}}

\newcommand{\textd}{\mathrm{d}}
\newcommand{\dd}[2][]{\ensuremath{\tfrac{\textd^{#1}}{\textd#2}}}
\newcommand{\ddt}{\ensuremath{\dd{t}}}

\newcommand{\coord}[1]{\bm{#1}}
\newcommand{\dimStateSpace}{N}
\newcommand{\dimOutput}{n_{\vecOutput}}
\newcommand{\dimRed}{n}

\newcommand{\mnf}{M}
\newcommand{\point}{x}
\newcommand{\tangentSpaceT}{T}
\newcommand{\TpMnf}[1][\point]{\tangentSpaceT_{#1}\mnf}

\newcommand{\diracStruct}{\mathcal{D}}

\newcommand{\diracStructISO}{\diracStruct_{\mathrm{iso}}}
\newcommand{\diracStructES}{\diracStruct_{\mathrm{es}}}

\newcommand{\spaceFlow}{\mathcal{F}}
\newcommand{\spaceFlowRes}{\spaceFlow_{\mathrm{R}}}
\newcommand{\spaceFlowResState}{\spaceFlow_{\mathrm{R},1}}
\newcommand{\spaceFlowResIO}{\spaceFlow_{\mathrm{R},2}}
\newcommand{\spaceFlowIO}{\spaceFlow_{\mathrm{IO}}}
\newcommand{\vecFlowRes}{f_{\spaceFlowRes}}
\newcommand{\vecFlowResState}{f_{\spaceFlowResState}}
\newcommand{\vecFlowResIO}{f_{\spaceFlowResIO}}
\newcommand{\vecFlowIO}{f_{\spaceFlowIO}}
\newcommand{\vecEffortRes}{e_{\spaceFlowRes}}
\newcommand{\vecEffortResState}{e_{\spaceFlowResState}}
\newcommand{\vecEffortResIO}{e_{\spaceFlowResIO}}
\newcommand{\vecEffortIO}{e_{\spaceFlowIO}}
\newcommand{\vecEffort}{e}

\newcommand{\vecFlow}{f}
\renewcommand{\vec}{v}
\newcommand{\vecAlt}{w}

\newcommand{\vecState}{x}
\newcommand{\vecInput}{u}
\newcommand{\vecOutput}{y}

\newcommand{\Ham}{H}
\newcommand{\gradHam}{\tfrac{\partial \Ham}{\partial \vecState}}
\newcommand{\gradHamCoord}{\tfrac{\partial \coord{\Ham}}{\partial \coord{\vecState}}}

\newcommand{\vecSpaceState}{v}

\newcommand{\vecSpaceIO}{v_{\mathrm{IO}}}

\newcommand{\reduce}[1]{\check{#1}}
\newcommand{\robMat}{\bm{V}}

\newcommand{\vecStateRedCoord}{\reduce{\coord{\vecState}}}

\newcommand{\mapISO}{K}
\newcommand{\mapES}{L}

\newcommand{\spaceState}{\mathbb{V}}
\newcommand{\spaceStateAlt}{\mathbb{W}}
\newcommand{\spaceIO}{\mathbb{V}_\mathrm{IO}}
\newcommand{\approximate}[1]{\widetilde{#1}}
\newcommand{\spaceIOApprox}{\approximate{\mathbb{V}}_\mathrm{IO}}
\newcommand{\spaceStateApprox}{\approximate{\mathbb{V}}}
\newcommand{\vecStateApprox}{\approximate{\vecState}}
\newcommand{\vecInputApprox}{\approximate{\vecInput}}
\newcommand{\vecOutputApprox}{\approximate{\vecOutput}}
\newcommand{\opES}{\mathcal{C}}
\newcommand{\fES}{f}
\newcommand{\domainTime}{I_t}
\newcommand{\tInit}{t_0}
\newcommand{\tEnd}{t_{\text{end}}}
\newcommand{\prodDual}[3][\spaceFlow]{\left\langle {#2}, {#3} \right\rangle_{#1}}

\newcommand{\restrict}[2]{\evalField[#2]{#1}}

\newcommand{\res}[1]{\hat{#1}}
\newcommand{\Kstate}{J}
\newcommand{\KstateIO}{G}
\newcommand{\KIO}{N}
\newcommand{\KStoepselState}{T_{1}}
\newcommand{\KStoepselIO}{T_{2}}
\newcommand{\Wstate}{R}
\newcommand{\WstateRes}{\res{\Wstate}}
\newcommand{\WIO}{S}
\newcommand{\WIORes}{\res{\WIO}}
\newcommand{\WstateIO}{P}
\newcommand{\WstateIORes}{\res{\WstateIO}}
\newcommand{\Lstate}{\omega}
\newcommand{\LstateIO}{\gamma}
\newcommand{\LIO}{\mu}
\newcommand{\LStoepselState}{\tau_{1}}
\newcommand{\LStoepselIO}{\tau_{2}}
\newcommand{\dissipation}{\phi}
\newcommand{\dissipationRes}{\res{\dissipation}}
\newcommand{\Phistate}{\rho}
\newcommand{\PhistateRes}{\res{\Phistate}}
\newcommand{\PhiIO}{\sigma}
\newcommand{\PhiIORes}{\res{\PhiIO}}
\newcommand{\PhistateIO}{\pi}
\newcommand{\PhistateIORes}{\res{\PhistateIO}}
\newcommand{\isopH}{\textsf{iso-pH}\xspace}
\newcommand{\espH}{\textsf{es-pH}\xspace}
\newcommand{\mor}{\textsf{MOR}\xspace}
\newcommand{\pH}{\textsf{pH}\xspace}
\newcommand{\ebe}{\textsf{EBE}\xspace}
\newcommand{\pbe}{\textsf{PBE}\xspace}
\newcommand{\di}{\textsf{DI}\xspace}

\title{Energy-stable Port-Hamiltonian Systems}
\author[P.~Buchfink \and S.~Glas \and H.~Zwart]{Patrick Buchfink${}^{\dagger}$ \and Silke Glas${}^\dagger$ \and Hans Zwart${}^{\dagger,\ddagger}$}

\address{${}^{\dagger}$ Department of Applied Mathematics, University of Twente, P.O. Box 217, 7500 AE Enschede, The Netherlands}
\email{\{p.buchfink,s.m.glas,h.j.zwart\}@utwente.nl}
\address{${}^{\ddagger}$ Department of Mechanical Engineering, Eindhoven University of Technology, P.O. Box 513, 5600 MB Eindhoven, The Netherlands}
\date{\today}

\begin{document}
\renewcommand*{\arraystretch}{1.2}
\maketitle
\begin{quote}
\footnotesize
\vspace{-7pt}
\textsc{Abstract.}
    We combine energy-stable and port-Hamiltonian (\pH) systems to obtain \emph{energy-stable port-Hamiltonian (\espH) systems}. The idea is to extend the known energy-stable systems with an input--output port, which results in a \pH formulation.
    One advantage of the new \espH formulation is that it naturally preserves its \espH structure throughout discretization (in space and time) and model reduction. 
\end{quote}
\vspace{21pt}

{\footnotesize \textsc{Keywords:}
energy-stable systems,
port-Hamiltonian systems,
structure-preserving discretization,
model reduction
}
\section{Introduction}
\label{sec:intro}

Energy-based concepts for evolution problems are important to formulate interpretable and physically consistent models.
Such models typically come with a power and energy balance equation (\pbe and \ebe) and/or dissipation inequality (\di) bounding the system energy,
which can be used to guarantee properties like passivity, stability, and boundedness of solutions.
The two energy-based systems relevant in this paper are (a)~energy-stable \cite{Egger2021} and (b)~port-Hamiltonian systems (\pH) (see, e.g., \cite{vanderSchaft2014}).
While energy-stable systems offer a straight-forward framework for structure-preserving discretization and model reduction (\mor), \pH systems extend Hamiltonian systems by adding ports; thus allowing to formulate open and dissipative systems.
Structure-preservation in discretization and \mor is beneficial since it typically enables to transfer physical guarantees of the system to the discrete (or reduced) system.

In this work, we show that (a) and (b) can be combined to a new formalism, \emph{energy-stable port-Hamiltonian (\espH) systems}.
This new formalism combines the advantages of (a) and (b) resulting in energy-based systems equipped with an input--output port and a general framework for structure-preserving discretization and \mor.
We present the \espH\ system in three formulations; each with a separate goal:
(i) a matrix formulation in \Cref{subsec:matrix_formulation} to illustrate the idea,
(ii) a formulation on possibly infinite-dimensional space in \Cref{sec:banach_space}, extending energy-stable systems,
and, (iii) a formulation on Dirac structures in \Cref{sec:dirac_structure} to show how \espH systems can be understood as \pH systems.

The paper provides a general structural framework for modeling energy-based systems with an \ebe and/or \di. For such a general framework, it is hard to address questions of existence and uniqueness. Thus, we do not specify initial value or boundary conditions here, which must be addressed when looking at a particular problem formulation.

Structure-preserving techniques have also been considered for
Hamiltonian systems via the Dirac--Frenkel formalism, see, e.g., \cite[Sec.~2]{Lub08},
Hamiltonian PDEs, see, e.g., \cite{Bridges2006},
Hamiltonian field theory, see, e.g., \cite{Marsden1983},
and 
energy-stable systems \cite{Egger2021}.
Our new \espH\ formulation extends these works by adding an input--output port that is crucial to model forced systems in control theory.
Port-Hamiltonian systems are already equipped with an input--output port. Moreover, a multitude of discretization techniques and \mor algorithms have been presented over the last years, see, e.g. \cite{polyuga2012effort, chaturantabut2016structure, https://doi.org/10.1002/pamm.201900399, kinon2023port, mehrmann2023control, Rettberg2023, SCHWERDTNER2023105655, Giesselmann2024}.
To our knowledge, however, none of these techniques are as natural as the techniques presented for energy-stable systems in \cite{Egger2021} in the sense that solely a (Petrov--)Galerkin projection is sufficient for structure-preservation, which is our motivation to formulate \espH systems.

As we understand our work as an extension to the related works listed in the previous paragraph, we foresee \espH systems to be applicable to the applications listed therein including, e.g., wave and elasticity equations~\cite{Marsden1983}, shallow water equations~\cite{Bridges2006}, and magneto-quasistatics~\cite{Egger2021}. A numerical study is subject to future work.

Notationwise, for a vector space $\spaceState$ and an interval $\domainTime = [\tInit, \tEnd] \subset \R$,
we denote an arbitrary point $\vecState \in \spaceState$ or a curve $\vecState(\cdot): \domainTime \to \spaceState$ using the same symbol $\vecState$. To differentiate the two objects, the curve is always indicated by brackets with $\vecState(\cdot)$ or $\vecState(t)$ for $t \in \domainTime$.
Moreover, we denote the evaluation of state-dependent operators with $\evalField{(\cdot)}$.
Note that the same notation may be used to denote the restriction of an operator to a subset/subspace $\spaceStateApprox \subset \spaceState$ with $\evalField[\spaceStateApprox]{(\cdot)}$.

\subsection{Illustrating the main idea in finite-dimensional vector spaces}
\label{subsec:matrix_formulation}
We are interested in energy-based systems using an \emph{energy functional} (or \emph{Hamiltonian}) $\coord{\Ham} \in \diffFun{1} (\R^{\dimStateSpace}, \R)$,
that describe the dynamics of
a state $\coord{\vecState}(\cdot): \domainTime \to \R^{\dimStateSpace}$ and
an output curve $\coord{\vecOutput}(\cdot): \domainTime \to \R^{\dimOutput}$
on a time interval $\domainTime \vcentcolon= [\tInit, \tEnd] \subset \R$.
As reference, we consider the \emph{input--state--output port-Hamiltonian (\isopH) system (with feedthrough term)}
\begin{linenomath*}
    \begin{equation}
    \label{eq:isopH_state_coord}
    \begin{aligned}
        \ddt \coord{\coord{\vecState}}(t)
&= \left( \evalField[\coord{\vecState}(t)]{\coord{\Kstate}} - \evalField[\coord{\vecState}(t)]{\coord{\Wstate}} \right) \evalField[\coord{\vecState}(t)]{\gradHamCoord}
+ \left( \evalField[\coord{\vecState}(t)]{\coord{\KstateIO}} - \evalField[\coord{\vecState}(t)]{\coord{\WstateIO}} \right) \coord{\vecInput}(t),\\
        \coord{\vecOutput}(t)
&= \left( \rT{\evalField[\coord{\vecState}(t)]{\coord{\KstateIO}}} + \rT{\evalField[\coord{\vecState}(t)]{\coord{\WstateIO}}} \right) \evalField[\coord{\vecState}(t)]{\gradHamCoord}
+ \left( \evalField[\coord{\vecState}(t)]{\coord{\WIO}} - {\evalField[\coord{\vecState}(t)]{\coord{\KIO}}} \right) \coord{\vecInput}(t),
\end{aligned}
\end{equation}
\end{linenomath*}
from, e.g., \cite{vanderSchaft2014,mehrmann2023control}, with matrix-valued functions $\coord{J}, \coord{R}: \R^{\dimStateSpace}  \to \R^{\dimStateSpace \times \dimStateSpace }$, $\coord{G}, \coord{P}: \R^{\dimStateSpace}  \to \R^{\dimStateSpace \times \dimOutput}$ and $\coord{S}, \coord{N}: \R^{\dimStateSpace}  \to \R^{\dimOutput \times \dimOutput }$.  
Moreover, it must hold, that the matrix-valued functions $\coord{Z}, \coord{W}: \R^{\dimStateSpace} \to \R^{(\dimStateSpace + \dimOutput) \times (\dimStateSpace + \dimOutput)}$ combining $\coord{J}$, $\coord{R}$, $\coord{G}$, $\coord{P}$, $\coord{S}$, and $\coord{N}$ are pointwise skew-symmetric and pointwise symmetric positive (semi-)definite, respectively,
\begin{linenomath*}
\begin{align*}
    \evalField[\coord{\vecState}]{\coord{Z}} &\vcentcolon=
    \begin{bmatrix}
        \evalField[\coord{\vecState}]{\coord{\Kstate}} & 
        \evalField[\coord{\vecState}]{\coord{\KstateIO}}\\
        -\rT{\evalField[\coord{\vecState}]{\coord{\KstateIO}}} & 
        \evalField[\coord{\vecState}]{\coord{\KIO}}
    \end{bmatrix},&
    \evalField[\coord{\vecState}]{\coord{W}} &\vcentcolon=
    \begin{bmatrix}
        \evalField[\coord{\vecState}]{\coord{\Wstate}} & 
        \evalField[\coord{\vecState}]{\coord{\WstateIO}} \\
        \rT{\evalField[\coord{\vecState}]{\coord{\WstateIO}}} & 
        \evalField[\coord{\vecState}]{\coord{\WIO}}
    \end{bmatrix},&
\end{align*}
\end{linenomath*}
allowing to rewrite \eqref{eq:isopH_state_coord} as
\begin{linenomath*}
\begin{align*}
    \begin{bmatrix}
        -\ddt \coord{\vecState}(t)\\
        \coord{\vecOutput}(t)
    \end{bmatrix}
    +
    \left(
        \evalField[\coord{\vecState}(t)]{\coord{Z}}
        - \evalField[\coord{\vecState}(t)]{\coord{W}}
    \right)
    \begin{bmatrix}
        \evalField[\coord{\vecState}(t)]{\gradHamCoord}\\
        \coord{\vecInput}(t)
    \end{bmatrix}
    =
    \coord{0}.
\end{align*}
\end{linenomath*}
In this paper, we introduce the \emph{energy-stable port-Hamiltonian (\espH) system}
\begin{linenomath*}
    \begin{equation}
    \label{eq:espH_state_coord}
    \begin{aligned}
        \left(
            -\evalField[\coord{\vecState}(t)]{\coord{\Lstate}}
            + \evalField[\coord{\vecState}(t)]{\coord{\Phistate}}
        \right) \ddt \coord{\vecState}(t)
    &= -\evalField[\coord{\vecState}(t)]{\gradHamCoord}
    + \left(
        \evalField[\coord{\vecState}(t)]{\coord{\LstateIO}}
        - \evalField[\coord{\vecState}(t)]{\coord{\PhistateIO}}
    \right) \coord{\vecInput}(t),\\
            \coord{\vecOutput}(t)
    &= \left( 
        \rT{\evalField[\coord{\vecState}(t)]{\coord{\LstateIO}}}
        + \rT{\evalField[\coord{\vecState}(t)]{\coord{\PhistateIO}}}
    \right) \ddt \coord{\vecState}(t)
    + \left(
        - \evalField[\coord{\vecState}(t)]{\coord{\LIO}}
        +{\evalField[\coord{\vecState}(t)]{\coord{\PhiIO}}}
    \right) \coord{\vecInput}(t),
    \end{aligned}
    \end{equation}
\end{linenomath*}
where $\coord{\omega}, \coord{\rho}: \R^{\dimStateSpace}  \to \R^{\dimStateSpace \times \dimStateSpace }$, $\coord{\gamma}, \coord{\pi}: \R^{\dimStateSpace}  \to \R^{\dimStateSpace \times \dimOutput}$ and \mbox{$\coord{\mu}, \coord{\sigma}: \R^{\dimStateSpace}  \to \R^{\dimOutput \times \dimOutput }$} are matrix-valued maps.
Again, it must hold that the matrix-valued maps ${\coord{\lambda}}, \coord{\dissipation}: \R^{\dimStateSpace}  \to \R^{(\dimStateSpace + \dimOutput) \times (\dimStateSpace + \dimOutput)}$ combining $\coord{\omega}$, $\coord{\rho}$, $\coord{\gamma}$, $\coord{\pi}$, $\coord{\mu}$, and $\coord{\sigma}$ are pointwise skew-symmetric
and symmetric positive (semi-)definite, respectively,
\begin{linenomath*}
\begin{align*}
    \evalField[\coord{\vecState}]{\coord{\lambda}} &\vcentcolon=
    \begin{bmatrix}
        \evalField[\coord{\vecState}]{\coord{\Lstate}} & 
        \evalField[\coord{\vecState}]{\coord{\LstateIO}}\\
        -\rT{\evalField[\coord{\vecState}]{\coord{\LstateIO}}} & 
        \evalField[\coord{\vecState}]{\coord{\LIO}}
    \end{bmatrix},&
    \evalField[\coord{\vecState}]{\coord{\dissipation}} &\vcentcolon=
    \begin{bmatrix}
        \evalField[\coord{\vecState}]{\coord{\Phistate}} & 
        \evalField[\coord{\vecState}]{\coord{\PhistateIO}} \\
        \rT{\evalField[\coord{\vecState}]{\coord{\PhistateIO}}} & 
        \evalField[\coord{\vecState}]{\coord{\PhiIO}}
    \end{bmatrix},
\end{align*}
\end{linenomath*}
{allowing to rewrite \eqref{eq:espH_state_coord} as }
\begin{linenomath*}
\begin{align*}
    \begin{bmatrix}
        -\evalField[\coord{\vecState}(t)]{\gradHamCoord}\\
        \coord{\vecOutput}(t)
    \end{bmatrix}
    + \left(
        \evalField[\coord{\vecState}(t)]{\coord{\lambda}}
        - \evalField[\coord{\vecState}(t)]{\coord{\dissipation}}
    \right)
\begin{bmatrix}
    \ddt \coord{\vecState}(t)\\
    \coord{\vecInput}(t)
\end{bmatrix}
&= \coord{0}.
\end{align*}
\end{linenomath*}
We show that the \espH\ system is a \pH system and
that it provides a \di.
The latter can be conducted from the previous equation by multiplication with
${
    \left[
        \rT{\ddt \coord{\vecState}(t)}\quad
        \rT{\coord{\vecInput}(t)}
    \right]
}$,
using the chain rule, skew-symmetry of $\evalField[\coord{\vecState}]{\coord{\lambda}}$ and positive (semi-)definiteness of $\evalField[\coord{\vecState}]{\coord{\dissipation}}$ with
\begin{linenomath*}
\begin{align*}
    \ddt \left[ \coord{\Ham}(\coord{\vecState}(t)) \right]
&= \rTb{\ddt \coord{\vecState}(t)} \evalField[\coord{\vecState}(t)]{\gradHamCoord}\\
&= \rT{\coord{\vecOutput}(t)} \coord{\vecInput}(t)
+
\underbrace{
    \rT{
        \begin{bmatrix}
            \ddt \coord{\vecState}(t)\\
            \coord{\vecInput}(t)
        \end{bmatrix}
    }
    \evalField[\coord{\vecState}(t)]{\coord{\lambda}}
    \begin{bmatrix}
        \ddt \coord{\vecState}(t)\\
        \coord{\vecInput}(t)
    \end{bmatrix}
}_{
    =0
}
\underbrace{
    -
    \rT{
        \begin{bmatrix}
            \ddt \coord{\vecState}(t)\\
            \coord{\vecInput}(t)
        \end{bmatrix}
    }
    \evalField[\coord{\vecState}(t)]{\coord{\dissipation}}
    \begin{bmatrix}
        \ddt \coord{\vecState}(t)\\
        \coord{\vecInput}(t)
    \end{bmatrix}
}_{
    \leq 0
}\\
&\leq \rT{\coord{\vecOutput}(t)} \coord{\vecInput}(t).
\end{align*}
\end{linenomath*}
Although the differences between the two formulations \eqref{eq:isopH_state_coord} and \eqref{eq:espH_state_coord} are subtle, the novel formulation \eqref{eq:espH_state_coord} comes with the key advantage that the formulation offers natural structure-preserving methods for discretization (in space and time) and \mor following the idea of the original energy-stable system \cite{Egger2021}.
To demonstrate this, we provide the formula for the reduced model in matrix formulation:%
\footnote{%
The same framework works for \mor of parametric systems but is skipped in the present paper for the sake of brevity.%
}
Galerkin projection with a reduced-order basis matrix $\robMat \in \R^{\dimStateSpace \times \dimRed}$ results in
\begin{linenomath*}
    \begin{equation}\label{eq:espH_coord_mor}
    \begin{aligned}
        \rT\robMat 
        \left(
            -\evalField[\robMat {\vecStateRedCoord}(t)]{\coord{\Lstate}}
            + \evalField[\robMat {\vecStateRedCoord}(t)]{\coord{\Phistate}}
        \right) \robMat \ddt {\vecStateRedCoord}(t)
    = -\rT\robMat \evalField[\robMat {\vecStateRedCoord}(t)]{\gradHamCoord}
    +
    \rT\robMat
    \left(
        \evalField[\robMat {\vecStateRedCoord}(t)]{\coord{\LstateIO}}
        - \evalField[\robMat {\vecStateRedCoord}(t)]{\coord{\PhistateIO}}
    \right) \coord{\vecInput}(t),\\
            \coord{\vecOutput}(t)
    =
    \left( 
        \rT{\evalField[\robMat {\vecStateRedCoord}(t)]{\coord{\LstateIO}}}
        + \rT{\evalField[\robMat {\vecStateRedCoord}(t)]{\coord{\PhistateIO}}}
    \right)
    \robMat \ddt {\vecStateRedCoord}(t)
    + \left(
        - \evalField[\robMat {\vecStateRedCoord}(t)]{\coord{\LIO}}
        +{\evalField[\robMat {\vecStateRedCoord}(t)]{\coord{\PhiIO}}}
    \right) \coord{\vecInput}(t),
    \end{aligned}
    \end{equation}
\end{linenomath*}
which is again an \espH\ system of reduced order $\dimRed \ll \dimStateSpace$ based on the reduced quantities $\evalField[\vecStateRedCoord]{\reduce{\coord{\Lstate}}} \vcentcolon= \rT\robMat \evalField[\robMat \vecStateRedCoord]{{\coord{\Lstate}}} \robMat \in \R^{\dimRed \times \dimRed}$,
$\evalField[\vecStateRedCoord]{\reduce{\coord{\Phistate}}} \vcentcolon= \rT\robMat \evalField[\robMat \vecStateRedCoord]{{\coord{\Phistate}}} \robMat \in \R^{\dimRed \times \dimRed}$,
$\reduce{\coord{\Ham}}(\reduce{\coord{\vecState}}) \vcentcolon= \coord{\Ham}(\robMat \reduce{\coord{\vecState}})$,
$\evalField[\vecStateRedCoord]{\reduce{\coord{\LstateIO}}} \vcentcolon= \rT\robMat \evalField[\robMat \vecStateRedCoord]{{\coord{\LstateIO}}}$,
$\evalField[\vecStateRedCoord]{\reduce{\coord{\PhistateIO}}} \vcentcolon= \rT\robMat \evalField[\robMat \vecStateRedCoord]{{\coord{\PhistateIO}}},
\in \R^{\dimRed \times \dimOutput}$ defined for all $\vecStateRedCoord \in \R^{\dimRed}$.%

\section{Formulation on Banach Spaces: Structure-preserving discretization and reduction via (Petrov-)Galerkin}
\label{sec:banach_space}

In this section, we extend energy-stable systems by an input--output port and derive the corresponding \pbe and \di. Moreover, the ideas of structure-preserving discretization and \mor are transferred from energy-stable systems to the \espH formulation.

\subsection{Fundamentals: Banach spaces, bounded linear operators and bilinear forms}
\label{sec:background}
Given two real Banach spaces $\left( \spaceState, \norm{\spaceState}{\cdot} \right)$, $\left( \spaceStateAlt, \norm{\spaceStateAlt}{\cdot} \right)$,
we denote the \emph{bounded linear operators from $\spaceState$ to $\spaceStateAlt$} with $\linBoundedOps(\spaceState, \spaceStateAlt)$.
Further, we denote the \emph{dual space} with $\dual\spaceState \vcentcolon= \linBoundedOps(\spaceState, \field)$ and the \emph{duality pairing} with $\prodDual[\spaceState]{\cdot}{\cdot}: \dual\spaceState \times \spaceState \to \field$.

Let $\linBoundedBilin(\spaceState, \spaceStateAlt; \field)$  be the space of bounded bilinear forms $a: \spaceState \times \spaceStateAlt \to \R$.
A bounded bilinear form $a \in \linBoundedBilin(\spaceState, \spaceStateAlt; \field)$ naturally defines two bounded linear maps $A \in \linBoundedOps(\spaceState, \dual\spaceStateAlt)$ and $\banachAdjoint{A} \in \linBoundedOps(\spaceStateAlt, \dual\spaceState)$ with
\begin{linenomath*} 
    \begin{align*}
        \forall \vec \in \spaceState, \vecAlt \in \spaceStateAlt:
        \prodDual[\spaceStateAlt]{A \vec}{\vecAlt}&
        \vcentcolon= a(\vec, \vecAlt), &
        &\text{and}&
        \forall \vec \in \spaceState, \vecAlt \in \spaceStateAlt:
        \prodDual[\spaceState]{\banachAdjoint{A}\vecAlt}{\vec}&
        \vcentcolon= a(\vec, \vecAlt)
    \end{align*}
\end{linenomath*}
and vice versa.
It holds $\prodDual[\spaceStateAlt]{A \vec}{\vecAlt} = a(\vec, \vecAlt) = \prodDual[\spaceState]{\banachAdjoint{A} \vecAlt}{\vec}$ for all $\vec \in \spaceState$, $\vecAlt \in \spaceStateAlt$.
In the case $\spaceState = \spaceStateAlt$, the bilinear form $a$ and respective operators $A$ and $\banachAdjoint{A}$ are called \emph{symmetric} if $a(\vec_2, \vec_1) = a(\vec_1, \vec_2)$ for all $\vec_1, \vec_2 \in \spaceState$ (or analogously $A = \banachAdjoint{A}$) and \emph{skew-symmetric} if $a(\vec_2, \vec_1) = -a(\vec_1, \vec_2)$ for all $\vec_1, \vec_2 \in \spaceState$ (or analogously $A = -\banachAdjoint{A}$).
Moreover, we refer to these objects as \emph{positive (semi-)definite} if $a(\vec, \vec) \geq 0$ or $a(\vec, \vec) > 0$ for all $\vec \in \spaceState$.

\subsection{Fundamentals: Energy-stable systems}

An \emph{energy-stable system} $(\spaceState, \opES, \Ham, \fES)$ is characterized by
a Banach space ($\spaceState$, $\norm{\spaceState}{\cdot}$),
a function $\opES: \spaceState \to \linBoundedOps(\spaceState, \dual\spaceState)$ yielding a pointwise bounded linear operator $\evalField[\vecState]{\opES} \in \linBoundedOps(\spaceState, \dual\spaceState)$,
a continuously differentiable functional $\Ham \in \diffFun{1} (\spaceState, \field)$ called the \emph{energy functional} with derivative $\evalField[\vecState]{\gradHam} \in \dual\spaceState$ at $\vecState \in \spaceState$,
and an additional right-hand side term $\fES: \spaceState \to \dual\spaceState$.
Then, the goal of an energy-stable system is to find $\vecState(\cdot) \in \diffFun{1}(\domainTime, \spaceState)$ of the evolution problem
\begin{linenomath*}
\[
\evalField[{\vecState(t)}]{\opES} \; \ddt \vecState(t)
= -\evalField[{\vecState(t)}]{\gradHam}
+ \evalField[\vecState(t)]{\fES}
\in \dual\spaceState
\]
\end{linenomath*}
or equivalently in the \emph{variational formulation}
\begin{linenomath*}
\begin{align}\label{eq:es_variational_coord}
    \prodDual[\spaceState]{\evalField[{\vecState(t)}]{\opES} \; \ddt \vecState(t)}{v}
    = - \prodDual[\spaceState]{\evalField[{\vecState(t)}]{\gradHam}}{v}
    + \prodDual[\spaceState]{\evalField[\vecState(t)]{\fES}}{v}
    \qquad
    \forall v \in \spaceState.
\end{align}
\end{linenomath*}

Note, that this formalism can equally model systems of partial differential equations (PDEs), ordinary differential equations (ODEs), or differential algebraic equations (DAEs) mostly depending on the choice of $\spaceState$ and $\opES$.
By the chain rule and evaluating \eqref{eq:es_variational_coord} for $\vecSpaceState=\ddt \vecState(t)$, all classical solution curves $\vecState \in \diffFun{1}(\domainTime, \spaceState)$ satisfy the \emph{\pbe}

\begin{linenomath*}
\[
\ddt \left[ \Ham(\vecState(t)) \right]
=
\prodDual[\spaceState]{\evalField[\vecState(t)]{\gradHam}}{\ddt \vecState(t)}
=
\prodDual[\spaceState]{\evalField[\vecState(t)]{\fES}}{\ddt \vecState(t)}
-
\prodDual[\spaceState]{\evalField[\vecState(t)]{\opES} \ddt \vecState(t)}{\ddt \vecState(t)}.
\]
\end{linenomath*}
As shown in \cite{Egger2021} or in \Cref{subsec:projection,subsec:time_discretization} for the more general \espH systems,
a (Petrov--)Galerkin projection of the variational formulation can be used to naturally preserve the underlying structure throughout discretization and \mor.

\subsection{Banach-space formulation of \espH\ systems}
\label{subsec:formulation}
We characterize a \espH\ system $(\spaceState, \spaceIO, \lambda, \dissipation, \Ham)$ with real Banach spaces $(\spaceState, \norm{\spaceState}{\cdot})$, $(\spaceIO, \norm{\spaceIO}{\cdot})$,
a pointwise skew-symmetric map
$\lambda: \spaceState \to\linBoundedOps\left(\spaceState \times \spaceIO, \dual{\left(\spaceState \times \spaceIO\right)}\right)$,
a pointwise symmetric and positive (semi-)definite map
$\dissipation: \spaceState \to\linBoundedOps\left(\spaceState \times \spaceIO, \dual{\left(\spaceState \times \spaceIO\right)}\right)$,
and a continuously differentiable functional $\Ham \in \diffFun{1} (\spaceState, \field)$ called the \emph{energy functional} with derivative $\evalField[\vecState]{\gradHam} \in \dual\spaceState$.
Given an input signal $\vecInput(\cdot): \domainTime \to \spaceIO$,
the goal of a \espH\ system is to find a solution curve $\vecState(\cdot) \in \diffFun{1}(\domainTime, \spaceState)$ and output curve $\vecOutput(\cdot): \domainTime \to \spaceState$ of
\begin{linenomath*}
\begin{align}
    \label{eq:espH_state}
    \left(
        -\evalField[\vecState(t)]{\Lstate}
        + \evalField[\vecState(t)]{\Phistate}
    \right) \ddt \vecState(t)
&= -\evalField[\vecState(t)]{\gradHam}
+ \left(
    \evalField[\vecState(t)]{\LstateIO}
    - \evalField[\vecState(t)]{\PhistateIO}
\right) \vecInput(t) \in \dual\spaceState,\\
    \label{eq:espH_io}
        \vecOutput(t)
&= \left( 
    \banachAdjoint{\evalField[\vecState(t)]{\LstateIO}}
    + \banachAdjoint{\evalField[\vecState(t)]{\PhistateIO}}
\right) \ddt \vecState(t)
+ \left(
    - \evalField[\vecState(t)]{\LIO}
    +{\evalField[\vecState(t)]{\PhiIO}}
\right) \vecInput(t) \in \dual\spaceIO,
\end{align}
\end{linenomath*}
where the maps $\omega, \rho: \spaceState \to\linBoundedOps\left(\spaceState, \dual{\spaceState}\right)$, $\gamma, \pi: \spaceState \to\linBoundedOps\left(\spaceIO, \dual{\spaceState}\right)$ and $\mu, \sigma: \spaceState \to\linBoundedOps\left( \spaceIO, \dual{ \spaceIO}\right)$ are subblocks of
\begin{linenomath*}
\begin{align*}
    \evalField[\vecState]{\lambda} &=
    \begin{bmatrix}
        \evalField[\vecState]{\Lstate} & 
        \evalField[\vecState]{\LstateIO}\\
        -\banachAdjoint{\evalField[\vecState]{\LstateIO}} & 
        \evalField[\vecState]{\LIO}
    \end{bmatrix},&
    &\text{and}&
    \evalField[\vecState]{\dissipation} &=
    \begin{bmatrix}
        \evalField[\vecState]{\Phistate} & 
        \evalField[\vecState]{\PhistateIO} \\
        \banachAdjoint{\evalField[\vecState]{\PhistateIO}} & 
        \evalField[\vecState]{\PhiIO}
    \end{bmatrix}.
\end{align*}
\end{linenomath*}
So \eqref{eq:espH_state} and \eqref{eq:espH_io} can be rewritten equivalently as
\begin{linenomath*}
\begin{align}
    \label{eq:espH_system}
        \begin{bmatrix}
            -\evalField[\vecState(t)]{\gradHam}\\
            \vecOutput(t)
        \end{bmatrix}
        +
        \left(
            \evalField[\vecState(t)]{\lambda} - \evalField[\vecState(t)]{\dissipation}
        \right)
        \begin{bmatrix}
            \ddt \vecState(t)\\
            \vecInput(t)
        \end{bmatrix}
    =
    0
    \in \dual{\spaceState} \times \dual{\spaceIO}.
\end{align}
\end{linenomath*}
With the choice $\evalField[\vecState(t)]{\opES} = -\evalField[\vecState(t)]{\Lstate} + \evalField[\vecState(t)]{\Phistate}$ and neglecting the input and output, the \espH system is an energy-stable system.
In parallel to \eqref{eq:es_variational_coord}, a variational formulation which is equivalent to \eqref{eq:espH_state} and \eqref{eq:espH_io} or \eqref{eq:espH_system} can be formulated as
\begin{linenomath*}
    \begin{align}\label{eq:espH_variational}
        \prodDual[\spaceState \times \spaceIO]{
            \begin{bmatrix}
                -\evalField[\vecState(t)]{\gradHam}\\
                \vecOutput(t)
            \end{bmatrix}
            +
            \left(
                \evalField[\vecState(t)]{\lambda} - \evalField[\vecState(t)]{\dissipation}
            \right)
            \begin{bmatrix}
                \ddt \vecState(t)\\
                \vecInput(t)
            \end{bmatrix}
        }{
            \begin{bmatrix}
                \vecSpaceState\\
                \vecSpaceIO
            \end{bmatrix}
        }
        &=
        0
        \qquad
        \forall \begin{bmatrix}
            \vecSpaceState\\
            \vecSpaceIO
        \end{bmatrix} \in \spaceState \times \spaceIO.
    \end{align}
    \end{linenomath*}
We now show that the \espH system allows to formulate a \pbe that respects the input--output ports and satisfies a \di.

\begin{theorem}\label{thm:power_balance_equation}
    The solution curve $\vecState(\cdot) \in \diffFun{1}(\domainTime, \spaceState)$ and output curve $\vecOutput(\cdot): \domainTime \to \spaceState$ of an \espH system obey the \pbe given in (a), and the \di given in (b),
    \begin{linenomath*}
    \begin{equation}
    \begin{aligned}\label{eq:dissipation_inequality}
        &\text{(a) }
        \ddt [\Ham(\vecState(t))]
        =
        -\prodDual[\spaceState \times \spaceIO]{
            \evalField[\vecState(t)]{\dissipation}
            \begin{bmatrix}
                \ddt \vecState(t)\\
                \vecInput(t)
            \end{bmatrix}
        }{
            \begin{bmatrix}
                \ddt \vecState(t)\\
                \vecInput(t)
            \end{bmatrix}
        }
        +
        \prodDual[\spaceIO]{\vecOutput(t)}{\vecInput(t)},\\
        &\text{(b) }
        \ddt [\Ham(\vecState(t))]
        \leq
        \prodDual[\spaceIO]{\vecOutput(t)}{\vecInput(t)}.
    \end{aligned}
    \end{equation}
    \end{linenomath*}
\end{theorem}
\begin{proof}
Evaluating \eqref{eq:espH_variational} with $\vecSpaceState = \ddt \vecState(t) \in \spaceState$ and $\vecSpaceIO = \vecInput(t) \in \spaceIO$,
it holds due to
skew-symmetry of $\evalField[\vecState(t)]{\lambda}$,
positive (semi-)definiteness of $\evalField[\vecState(t)]{\dissipation}$, and
$\ddt [\Ham(\vecState(t))] = \prodDual[\spaceState]{\evalField[\vecState(t)]{\gradHam}}{\ddt \vecState(t)}$ that
\begin{linenomath*}
\begin{align*}
    \ddt [\Ham(\vecState(t))]
=&
\underbrace{
    \prodDual[\spaceState \times \spaceIO]{
        \evalField[\vecState(t)]{\lambda}
        \begin{bmatrix}
            \ddt \vecState(t)\\
            \vecInput(t)
        \end{bmatrix}
    }{
        \begin{bmatrix}
            \ddt \vecState(t)\\
            \vecInput(t)
        \end{bmatrix}
    }
}_{
    = 0
}\\
&\underbrace{
    -\prodDual[\spaceState \times \spaceIO]{
        \evalField[\vecState(t)]{\dissipation}
        \begin{bmatrix}
            \ddt \vecState(t)\\
            \vecInput(t)
        \end{bmatrix}
    }{
        \begin{bmatrix}
            \ddt \vecState(t)\\
            \vecInput(t)
        \end{bmatrix}
    }
}_{
    \leq 0
}\\
&+
\prodDual[\spaceIO]{\vecOutput(t)}{\vecInput(t)}\\
\leq&
\prodDual[\spaceIO]{\vecOutput(t)}{\vecInput(t)}.\qedhere
\end{align*}
\end{linenomath*}
\end{proof}
Integrating the \pbe (\ref{eq:dissipation_inequality}a) over time yields the \emph{energy balance equation} (\ebe)
\begin{linenomath*}
\begin{align}\label{eq:energy_equations}
    \Ham(\vecState(t))
    - \Ham(\vecState(s))
    =
    \int_{s}^t
    -\prodDual[\spaceState \times \spaceIO]{
        \evalField[\vecState(\tau)]{\dissipation}
        \begin{bmatrix}
            \ddt \vecState(\tau)\\
            \vecInput(\tau)
        \end{bmatrix}
    }{
        \begin{bmatrix}
            \ddt \vecState(\tau)\\
            \vecInput(\tau)
        \end{bmatrix}
    }
    +
    \prodDual[\spaceIO]{\vecOutput(\tau)}{\vecInput(\tau)} \intd \tau.
\end{align}
\end{linenomath*}

\subsection{Structure-preserving spatial discretization and model reduction via Galerkin projection}
\label{subsec:projection}

We assume two closed subspaces $\spaceStateApprox \subset \spaceState$, $\spaceIOApprox \subset \spaceIO$ as well as input $\vecInputApprox(\cdot): \domainTime \to \spaceIOApprox$ are given. Then, the solution $\vecStateApprox(\cdot) \in \diffFun{1} (\domainTime, \spaceStateApprox)$ and output $\vecOutputApprox(\cdot) \in \diffFun{1}(\domainTime, \spaceIOApprox)$ are determined by the variational formulation \eqref{eq:espH_variational} on the subspaces, i.e., for all $\vecSpaceState \in \spaceStateApprox$, $\vecSpaceIO \in \spaceIOApprox$, which results in an \espH system
$
\left(
    \spaceStateApprox, \spaceIOApprox,
    \restrict{\lambda}{\spaceStateApprox \times \spaceIOApprox},
    \restrict{\dissipation}{\spaceStateApprox \times \spaceIOApprox}
    \restrict{\Ham}{\spaceStateApprox}
\right)
$.
Since the Galerkin projection results again in an \espH system, it is called \emph{structure-preserving}.

The Galerkin projection, in general, is well known to formulate spatial discretization techniques for PDEs and \mor.
Combining the \espH formulation with a Galerkin projection leads to a structure-preserving projection which  
yields the advantage that the Galerkin approximation inherits the properties from \Cref{thm:power_balance_equation}.
Writing a finite-dimensional system
$\left(
    \spaceStateApprox, \spaceIOApprox,
    \restrict{\lambda}{\spaceStateApprox \times \spaceIOApprox},
    \restrict{\dissipation}{\spaceStateApprox \times \spaceIOApprox}
    \restrict{\Ham}{\spaceStateApprox}
\right)$
in a basis results in the reduced model \eqref{eq:espH_coord_mor} given in \Cref{subsec:matrix_formulation}.

\subsection{Structure-preserving temporal discretization via Petrov--Galerkin projection}
\label{subsec:time_discretization}
Structure-preserving methods for time discretization can also be obtained within the energy-stable formalism.
The following techniques are formulated for an \espH system $\left(\spaceState, \spaceIO, \lambda, \dissipation, \Ham\right)$ but also hold for systems semi-discretized in space or reduced systems as described in the previous section.
We follow the construction of a time-discretized problem from the energy-stable systems \cite{Egger2021} via a Petrov--Galerkin projection and extend the ideas to the \espH system. 

Consider finite-dimensional subspaces $\spaceState_N \subset H^1(\domainTime; \spaceState)$, $\spaceStateAlt_N \subset L_2(\domainTime; \spaceState)$, $\spaceState_{\mathrm{IO},N}, \spaceStateAlt_{\mathrm{IO},N} \subset L_2(\domainTime; \spaceIO)$, and $\spaceState_{y,N} \subset L_2(\domainTime; \dual\spaceIO)$.
Then the time-discretized problem is given by the Petrov--Galerkin projection:
Given an input curve $\vecInput_N(\cdot) \in \spaceState_{\mathrm{IO},N}$,
find a solution curve $\vecState_N(\cdot) \in \spaceState_N$ and an output curve $\vecOutput_N(\cdot) \in \spaceState_{y,N}$ such that
\begin{linenomath*}
    \begin{align}\label{eq:time_discrete}
        \int_{\domainTime}
        \prodDual[\spaceState \times \spaceIO]{
            \begin{bmatrix}
                -\evalField[\vecState(t)]{\gradHam}\\
                \vecOutput(t)
            \end{bmatrix}
            +
            \left(
                \evalField[\vecState(t)]{\lambda} - \evalField[\vecState(t)]{\dissipation}
            \right)
            \begin{bmatrix}
                \ddt \vecState(t)\\
                \vecInput(t)
            \end{bmatrix}
        }{
            \begin{bmatrix}
                \vecSpaceState(t)\\
                \vecSpaceIO(t)
            \end{bmatrix}
        }
        \intd{t}
        &=
        0
    \end{align}
\end{linenomath*}
for all $(\vecSpaceState(\cdot), \vecSpaceIO(\cdot)) \in \spaceStateAlt_{N} \times \spaceStateAlt_{\mathrm{IO},N}$.
In \cite[Remark~5 and 6]{Egger2021}, it is shown for energy-stable systems that
this formulation covers discrete gradient methods and the average vector field
collocation method and is closely related to the Lobatto-IIIA method.
The present formulation allows to extend those ideas to systems with an input--output port.

If $\ddt \vecState_N(\cdot) \in \spaceStateAlt_N$ and $\vecOutput(\cdot) \in \spaceStateAlt_{\mathrm{IO},N}$, it can be shown similarly to the proof of \Cref{thm:power_balance_equation} that an \ebe like \eqref{eq:energy_equations} and an integrated \di hold:
\begin{linenomath*}
\begin{align*}
    \Ham\left(\vecState_N(\tEnd)\right) - \Ham\left(\vecState_N(\tInit)\right)
=&
    -
    \int_{\domainTime}
    \prodDual[\spaceState \times \spaceIO]{
        \evalField[\vecState(t)]{\dissipation}
        \begin{bmatrix}
            \ddt \vecState(t)\\
            \vecInput(t)
        \end{bmatrix}
    }{
        \begin{bmatrix}
            \ddt \vecState(t)\\
            \vecInput(t)
        \end{bmatrix}
    }
    \intd t\\
&+
    \int_{\domainTime}
    \prodDual[\spaceIO]{\vecOutput_N(t)}{\vecInput_N(t)}
    \intd t\\
\leq& \int_{\domainTime}
    \prodDual[\spaceIO]{\vecOutput_N(t)}{\vecInput_N(t)}
\intd t.
\end{align*}
\end{linenomath*}

\section{Dirac structure formulation: Connecting to port-Hamiltonian systems}
\label{sec:dirac_structure}

In this section, we show that the \espH\ system is a \pH system. 
To this end, \Cref{subsec:dirac_structures} formulates general \pH systems on manifolds via Dirac structures. \Cref{subsec:iso_ph,subsec:es_pH} show how the established \isopH systems are classically derived in the Dirac formulation and how to transfer these ideas to \espH systems.
For the sake of brevity, we restrict to finite-dimensional systems.

\subsection{Fundamentals: Dirac structure and finite-dimensional port-Hamiltonian systems}
\label{subsec:dirac_structures}
Assume to be given a finite-dimensional real Banach space $\spaceFlow$.
In contrast to the previous sections, we use $\spaceFlow$ instead of $\spaceState$ to match the notation established for Dirac structures.
Moreover, an element
$\vecFlow \in \spaceFlow$ is called a \emph{flow vector},
$\vecEffort \in \dual\spaceFlow$ is called an \emph{effort vector},
and the dual product $\prodDual[\spaceFlow]{\vecEffort}{\vecFlow}$ describes a physical power.

We call a subspace $\diracStruct \subset \spaceFlow \times \dual\spaceFlow$ a \emph{Dirac structure} if
\begin{linenomath*}
\begin{align}\label{eq:dirac_struct}
    \text{(i)~} \prodDual{\vecEffort}{\vecFlow} &= 0 \quad \forall (\vecFlow, \vecEffort) \in \diracStruct&
    &\text{and}&
    \text{(ii)~} \dim(\diracStruct) &= \dim(\spaceFlow).
\end{align}
\end{linenomath*}
From the first property, (i), it can be shown that $\dim(\diracStruct) \leq \dim(\spaceFlow)$.
Thus, the interpretation of a Dirac structure is that it is a subspace of $\spaceFlow \times \dual\spaceFlow$ which is (i) power-conserving and (ii) of maximal dimension.

A \emph{\pH system (formulated on a finite-dimensional manifold)}
is characterized by a manifold $\mnf$, on which the dynamics evolve,
a continuously differentiable functional $\Ham \in \diffFun{1}(\mnf, \field)$
(called the \emph{energy functional} or \emph{Hamiltonian}),
finite-dimensional vector spaces $\spaceFlowIO$ and $\spaceFlowRes$ which define spaces for input--output and resistive variables,
and a state-dependent\footnote{An alternative naming is \emph{modulated} Dirac structure.} \emph{Dirac structure}
$\evalField[\vecState]{\diracStruct} \subset (\TpMnf \times \spaceFlowIO \times \spaceFlowRes) \times \dual{(\TpMnf \times \spaceFlowIO \times \spaceFlowRes)}$, i.e., $\evalField[\vecState]{\diracStruct}$ is for each fixed $\vecState \in \mnf$ a Dirac structure in the understanding of the previous paragraph with the flow space $\spaceFlow = \TpMnf \times \spaceFlowIO \times \spaceFlowRes$ where $\TpMnf$ denotes the tangent space of $\mnf$ at the point $\vecState \in \mnf$.
The duality pairing of the flow space $\spaceFlow = \TpMnf \times \spaceFlowIO \times \spaceFlowRes$ is assumed to be inherited from the duality pairings on $\TpMnf$, $\spaceFlowIO$, and $\spaceFlowRes$ with
\begin{linenomath*}
\begin{align}\label{eq:split_power}
  &\prodDual{\vecEffort}{\vecFlow}
  = \prodDual[\TpMnf]{\vecEffort_{\TpMnf}}{\vecFlow_{\TpMnf}}
  + \prodDual[\spaceFlowIO]{\vecEffort_{\spaceFlowIO}}{\vecFlow_{\spaceFlowIO}}
  + \prodDual[\spaceFlowRes]{\vecEffort_{\spaceFlowRes}}{\vecFlow_{\spaceFlowRes}},
\end{align}
\end{linenomath*}
for all $\vecFlow = \left(\vecFlow_{\TpMnf}, \vecFlow_{\spaceFlowIO}, \vecFlow_{\spaceFlowRes} \right) \in \spaceFlow$ and $\vecEffort = \left(\vecEffort_{\TpMnf}, \vecEffort_{\spaceFlowIO}, \vecEffort_{\spaceFlowRes} \right) \in \dual\spaceFlow$.
The \pH system determines the evolution of a \emph{state trajectory} $\vecState(\cdot) \in \diffFun{1}(\domainTime, \mnf)$ based on trajectories $\vecFlowIO(\cdot): \domainTime \to \spaceFlowIO$, $\vecEffortIO(\cdot): \domainTime \to \dual\spaceFlowIO$ for the input--output variables and $\vecFlowRes(\cdot): \domainTime \to \spaceFlowRes$, $\vecEffortRes(\cdot): \domainTime \to \dual\spaceFlowRes$ for the resistive variables by requiring that
\begin{linenomath*}
\begin{align}\label{eq:dynamics_in_dirac_struct}
    \left(
        -\ddt \vecState(t),
        \vecFlowIO(t),
        \vecFlowRes(t),
        \evalField[\vecState(t)]{\gradHam},
        \vecEffortIO(t),
        \vecEffortRes(t)
    \right) \in \diracStruct(\vecState(t)),
    \qquad
    \forall t \in \domainTime.
\end{align}
\end{linenomath*}
The idea of this choice is that it automatically follows from \eqref{eq:dirac_struct}, \eqref{eq:split_power}, \eqref{eq:dynamics_in_dirac_struct} that
\begin{linenomath*}
\begin{align*}
    \prodDual[\TpMnf]{\evalField[\vecState(t)]{\gradHam}}{-\ddt \vecState(t)}
    +
    \prodDual[\spaceFlowIO]{\vecEffortIO(t)}{\vecFlowIO(t)}
    +
    \prodDual[\spaceFlowRes]{\vecEffortRes(t)}{\vecFlowRes(t)}
    =
    0
    &
\end{align*}
and, thus, with the additional assumption that $\prodDual[\spaceFlowRes]{\vecEffortRes(t)}{\vecFlowRes(t)} \leq 0$,
and by using the chain rule, a \pbe and \di hold
\begin{equation}
\begin{aligned}\label{eq:pH_power_equations}
    \ddt\left[
        \Ham(\vecState(t))
    \right]
    =&
    \prodDual[\TpMnf]{\evalField[\vecState(t)]{\gradHam}}{\ddt \vecState(t)}\\
    =&
    \prodDual[\spaceFlowIO]{\vecEffortIO(t)}{\vecFlowIO(t)}
    +
        \prodDual[\spaceFlowRes]{\vecEffortRes(t)}{\vecFlowRes(t)}\\
    \leq&
    \prodDual[\spaceFlowIO]{\vecEffortIO(t)}{\vecFlowIO(t)}.
\end{aligned}
\end{equation}
\end{linenomath*}

\subsection{Fundamentals: Input--state--output port-Hamiltonian system based on Dirac structure formulation}
\label{subsec:iso_ph}
The \isopH system is a classical \pH formulation with flow space $\spaceFlow = \TpMnf \times \spaceFlowIO \times \spaceFlowRes$ (see, e.g., \cite{vanderSchaft2014,mehrmann2023control}).
The resistive variables are split in two parts $\spaceFlowRes = \spaceFlowResState \times \spaceFlowResIO$ which is designed to introduce a dissipation on the state via $\spaceFlowResState$ and the IO variables via $\spaceFlowResIO$.
The \isopH system is characterized by a pointwise skew-symmetric operator ${\mapISO}: \mnf \to \linBoundedOps(\dual\spaceFlow, \spaceFlow)$, which defines the Dirac structure $\evalField[\vecState]{\diracStructISO}$, and a pointwise symmetric, positive (semi-)definite operator ${\res{W}}: \mnf \to \linBoundedOps(\dual\spaceFlowRes, \spaceFlowRes)$ to describe a linear relation of the resistive variables with
\begin{linenomath*}
\begin{equation}
    \begin{aligned}\label{eq:dirac_struct_iso}
        \evalField[\vecState]{\diracStructISO}
    &= \left\{
        (\vecFlow, \vecEffort) \in
    \spaceFlow \times
    \dual\spaceFlow
    \,\big|\,
    \vecFlow = \evalField[\vecState]{\mapISO}\, \vecEffort
    \right\},\\
    \evalField[\vecState]{\mapISO}
&=\vcentcolon \begin{bmatrix}
    -\evalField[\vecState]{\Kstate} & -\evalField[\vecState]{\KstateIO} & -\evalField[\vecState]{\KStoepselState} & 0\\
    \banachAdjoint{\evalField[\vecState]{\KstateIO}} & -\evalField[\vecState]{\KIO} & 0 & -\evalField[\vecState]{\KStoepselIO}\\
    \banachAdjoint{\evalField[\vecState]{\KStoepselState}} & 0 & 0 & 0\\
    0 & \banachAdjoint{\evalField[\vecState]{\KStoepselIO}} & 0 & 0\\
\end{bmatrix},\\
    \vecEffortRes &=
-\underbrace{
    \begin{bmatrix}
        \evalField[\vecState]{\WstateRes} & \evalField[\vecState]{\WstateIORes}\\
        \banachAdjoint{\evalField[\vecState]{\WstateIORes}} & \evalField[\vecState]{\WIORes}\\
    \end{bmatrix}
}_{\vcentcolon= \evalField[\vecState]{\res{W}}} \vecFlowRes
\end{aligned}
\end{equation}
\end{linenomath*}
for all $\vecState \in \mnf$.
For a given {input trajectory} $\vecInput(\cdot): \domainTime \to \dual\spaceFlowIO$ with resulting {output trajectory} $\vecOutput(\cdot): \domainTime \to \spaceFlowIO$,
the choice of the dynamics in \eqref{eq:dynamics_in_dirac_struct} with $\vecFlowIO(\cdot) = \vecOutput(\cdot)$ and $\vecEffortIO(\cdot) = \vecInput(\cdot)$ together with \eqref{eq:dirac_struct_iso} results in the differential algebraic system
\begin{linenomath*}
\begin{align*}
    -\ddt \vecState(t)
&= -\evalField[\vecState(t)]{\Kstate} \evalField[\vecState(t)]{\gradHam}
- \evalField[\vecState(t)]{\KstateIO}\, \vecInput(t)
- \evalField[\vecState(t)]{\KStoepselState} \vecEffortResState(t),&
    \vecFlowResState(t)
&= \banachAdjoint{\evalField[\vecState(t)]{\KStoepselState}} \evalField[\vecState(t)]{\gradHam},\\
    \vecOutput(t)
&= \banachAdjoint{\evalField[\vecState(t)]{\KstateIO}} \evalField[\vecState(t)]{\gradHam}
- \evalField[\vecState(t)]{\KIO}\, \vecInput(t)
- \evalField[\vecState(t)]{\KStoepselIO} \vecEffortResIO(t),&
    \vecFlowResIO(t)
&= \banachAdjoint{\evalField[\vecState(t)]{\KStoepselIO}} \vecInput(t)
\end{align*}
\end{linenomath*}
with $\vecFlowRes = (\vecFlowResState, \vecFlowResIO)$ and $\vecEffortRes = (\vecEffortResState, \vecEffortResIO)$.
The algebraic equations for the resistive trajectories can be resolved by using $\vecEffortRes(t) = -\evalField[\vecState(t)]{\res{W}} \vecFlowRes(t)$ from \eqref{eq:dirac_struct_iso} which results in a system of ODEs with output equation
\begin{linenomath*}
\begin{align*}
        \ddt \vecState(t)
&= \left( \evalField[\vecState(t)]{\Kstate} - \evalField[\vecState(t)]{\Wstate} \right) \evalField[\vecState(t)]{\gradHam}
+ \left( \evalField[\vecState(t)]{\KstateIO} - \evalField[\vecState(t)]{\WstateIO} \right) \vecInput(t),\\
        \vecOutput(t)
&= \left( \banachAdjoint{\evalField[\vecState(t)]{\KstateIO}} + \banachAdjoint{\evalField[\vecState(t)]{\WstateIO}} \right) \evalField[\vecState(t)]{\gradHam}
+ \left( \evalField[\vecState(t)]{\WIO} - {\evalField[\vecState(t)]{\KIO}} \right) \vecInput(t),
\end{align*}
\end{linenomath*}
with $\evalField[\vecState]{\Wstate} \vcentcolon= \evalField[\vecState]{\KStoepselState} \evalField[\vecState]{\WstateRes} \banachAdjoint{\evalField[\vecState]{\KStoepselState}}$,
$\evalField[\vecState]{\WstateIO} \vcentcolon= \evalField[\vecState]{\KStoepselState} \evalField[\vecState]{\WstateIORes} \banachAdjoint{\evalField[\vecState]{\KStoepselIO}}$,
and
$\evalField[\vecState]{\WIO} \vcentcolon= \evalField[\vecState]{\KStoepselIO} \evalField[\vecState]{\WIORes} \banachAdjoint{\evalField[\vecState]{\KStoepselIO}}$.
Since $\evalField[\vecState]{\res{W}}$ is positive \mbox{(semi-)}definite, it holds $\prodDual[\spaceFlowRes]{\vecEffortRes(t)}{\vecFlowRes(t)}
=
-\prodDual[\spaceFlowRes]{\evalField[\vecState(t)]{\res{W}} \vecFlowRes(t)}{\vecFlowRes(t)}
\leq 0$; thus the \pbe and \di, see \eqref{eq:pH_power_equations}, hold.

\subsection{Energy-based port-Hamiltonian system based on Dirac structure formulation}
\label{subsec:es_pH}
To formulate an \espH\ system in the \pH framework based on Dirac structures,
we consider the flow space
$\spaceFlow
= \TpMnf \times \spaceFlowIO \times \spaceFlowRes$
with input--output and resistive variables
where $\spaceFlowRes = \spaceFlowResState \times \spaceFlowResIO$.
The \espH system is characterized by a pointwise skew-symmetric operator ${\mapES}: \mnf \to \linBoundedOps(\spaceFlow, \dual\spaceFlow)$, which defines the Dirac structure $\evalField[\vecState]{\diracStructES} \subset \spaceFlow \times \dual\spaceFlow$, and a pointwise symmetric, positive (semi-)definite operator ${\res{\dissipation}}: \mnf \to \linBoundedOps(\spaceFlowRes, \dual\spaceFlowRes)$ with
\begin{linenomath*}
\begin{equation}
\begin{aligned}\label{eq:dirac_struct_es}
\evalField[\vecState]{\diracStructES}
&\vcentcolon=
\left\{
    (\vecFlow, \vecEffort) \in \spaceFlow \times \dual\spaceFlow
    \;\big|\;
    \vecEffort = \evalField[\vecState]{\mapES} \vecFlow
\right\},\\
    \evalField[\vecState]{\mapES}
&=\vcentcolon
\begin{bmatrix}
    -\evalField[\vecState]{\Lstate} & \evalField[\vecState]{\LstateIO} & \evalField[\vecState]{\LStoepselState} & 0\\
    -\banachAdjoint{\evalField[\vecState]{\LstateIO}} & -{\evalField[\vecState]{\LIO}} & 0 & -{\evalField[\vecState]{\LStoepselIO}}\\
    -\banachAdjoint{\evalField[\vecState]{\LStoepselState}} & 0 & 0 &0\\
    0 & \banachAdjoint{\evalField[\vecState]{\LStoepselIO}} & 0 &0\\
\end{bmatrix},\\
    \vecFlowRes &=
- \underbrace{
    \begin{bmatrix}
        \evalField[\vecState]{\PhistateRes} & \evalField[\vecState]{\PhistateIORes}\\
        \banachAdjoint{\evalField[\vecState]{\PhistateIORes}} & \evalField[\vecState]{\PhiIORes}\\
    \end{bmatrix}
}_{\vcentcolon= \evalField[\vecState]{\dissipationRes}}
\vecEffortRes
\end{aligned}
\end{equation}
\end{linenomath*}
for all $\vecState \in \mnf$.
Assuming to be given an {input trajectory} $\vecInput(\cdot): \domainTime \to \spaceFlowIO$ and
to be interested in an {output trajectory} $\vecOutput(\cdot): \domainTime \to \dual\spaceFlowIO$,
the choice of the dynamics in \eqref{eq:dynamics_in_dirac_struct} with $\vecFlowIO(\cdot) = \vecInput(\cdot)$ and $\vecEffortIO(\cdot) = \vecOutput(\cdot)$ results in combination with \eqref{eq:dirac_struct_es} in the differential algebraic system
\begin{linenomath*}
\begin{align*}
    \evalField[\vecState(t)]{\gradHam}
&= \evalField[\vecState(t)]{\Lstate} \ddt \vecState(t)
+\evalField[\vecState(t)]{\LstateIO} \vecInput(t)
+\evalField[\vecState(t)]{\LStoepselState} \vecFlowResState(t),&
    \vecEffortResState(t)
&= \banachAdjoint{\evalField[\vecState(t)]{\LStoepselState}} \ddt \vecState(t),\\
    \vecOutput(t)
&= \banachAdjoint{\evalField[\vecState(t)]{\LstateIO}} \ddt \vecState(t)
- \evalField[\vecState(t)]{\LIO}\, \vecInput(t)
- \evalField[\vecState(t)]{\LStoepselIO}\, \vecFlowResIO(t),&
    \vecEffortResIO(t)
&= \banachAdjoint{\evalField[\vecState(t)]{\LStoepselIO}} \vecInput(t).
\end{align*}
\end{linenomath*}
The algebraic equations for the resistive trajectories can be resolved by using $\vecFlowRes(t) = -\evalField[\vecState(t)]{\dissipationRes} \vecEffortRes(t),$ from \eqref{eq:dirac_struct_es}, which results in an evolution problem with output equation
\begin{linenomath*}
\begin{equation*}
\begin{aligned}
    \left( -\evalField[\vecState(t)]{\Lstate} + \evalField[\vecState(t)]{\Phistate} \right) \ddt \vecState(t)
&= -\evalField[\vecState(t)]{\gradHam}
+ \left( \evalField[\vecState(t)]{\LstateIO} - \evalField[\vecState(t)]{\PhistateIO} \right) \vecInput(t),\\
        \vecOutput(t)
&= \left( 
    \banachAdjoint{\evalField[\vecState(t)]{\LstateIO}}
    + \banachAdjoint{\evalField[\vecState(t)]{\PhistateIO}}
\right) \ddt \vecState(t)
+ \left(
    {\evalField[\vecState(t)]{\PhiIO}}
    - \evalField[\vecState(t)]{\LIO}
\right) \vecInput(t),
\end{aligned}
\end{equation*}
\end{linenomath*}
where we abbreviate
$\evalField[\vecState]{\Phistate} \vcentcolon= \evalField[\vecState]{\LStoepselState}\, \evalField[\vecState]{\PhistateRes} \banachAdjoint{\evalField[\vecState]{\LStoepselState}}$,
$\evalField[\vecState]{\PhistateIO} \vcentcolon= \evalField[\vecState]{\LStoepselState}\, \evalField[\vecState]{\PhistateIORes} \banachAdjoint{\evalField[\vecState]{\LStoepselIO}}$, and
$\evalField[\vecState]{\PhiIO} \vcentcolon= \evalField[\vecState]{\LStoepselIO}\, \evalField[\vecState]{\PhiIORes} \banachAdjoint{\evalField[\vecState]{\LStoepselIO}}$.
Since $\evalField[\vecState]{\dissipationRes}$ is positive (semi-)definite, it holds $\prodDual[\spaceFlowRes]{\vecEffortRes(t)}{\vecFlowRes(t)}
=
-\prodDual[\spaceFlowRes]{\vecEffortRes(t)}{\evalField[\vecState(t)]{\dissipationRes} \vecEffortRes(t)}
\leq 0$
and, thus,
the \pbe and \di \eqref{eq:pH_power_equations} hold.
Both equations have already been proven for the infinite-dimensional case in \Cref{thm:power_balance_equation}. In this subsection, however, the equations have been linked to the theory of Dirac structures and \pH systems.

\section{Conclusion and Outlook}
\label{sec:conclusion_outlook}
We introduced \espH systems by extending energy-stable systems with a consistent output equation (\Cref{sec:banach_space}) and by linking it to \pH systems via Dirac structures (\Cref{sec:dirac_structure}). 
The advantage of the energy-stable formulation of \pH systems is that it allows a unified framework for structure-preserving discretization and for \mor.
Future work will investigate a numerical example using the full potential of the unified framework for structure-preserving discretization and for \mor.

\bibliographystyle{plain-doi}
\bibliography{references}

\begin{thebibliography}{10}

\bibitem{Bridges2006}
\textsc{T.~J. Bridges and S.~Reich}.
\newblock \href{https://doi.org/10.1088/0305-4470/39/19/s02}{Numerical methods for {H}amiltonian {PDEs}}.
\newblock {\em Journal of Physics A: Mathematical and General}, 39(19):5287--5320, apr 2006.

\bibitem{chaturantabut2016structure}
\textsc{S.~Chaturantabut, C.~Beattie, and S.~Gugercin}.
\newblock Structure-preserving model reduction for nonlinear port-{H}amiltonian systems.
\newblock {\em SIAM Journal on Scientific Computing}, 38(5):B837--B865, 2016.

\bibitem{Egger2021}
\textsc{H.~Egger, O.~Habrich, and V.~Shashkov}.
\newblock \href{https://doi.org/10.1515/cmam-2020-0025}{On the energy stable approximation of {H}amiltonian and gradient systems}.
\newblock {\em Computational Methods in Applied Mathematics}, 21(2):335--349, 2021.

\bibitem{Giesselmann2024}
\textsc{J.~Giesselmann, A.~Karsai, and T.~Tscherpel}.
\newblock Energy-consistent {P}etrov--{G}alerkin time discretization of port-{H}amiltonian systems, 2024.

\bibitem{kinon2023port}
\textsc{P.~L. Kinon, T.~Thoma, P.~Betsch, and P.~Kotyczka}.
\newblock Port-{H}amiltonian formulation and structure-preserving discretization of hyperelastic strings.
\newblock {\em arXiv preprint arXiv:2304.10957}, 2023.

\bibitem{https://doi.org/10.1002/pamm.201900399}
\textsc{B.~Liljegren-Sailer and N.~Marheineke}.
\newblock \href{https://doi.org/https://doi.org/10.1002/pamm.201900399}{Structure-preserving {G}alerkin approximation for a class of nonlinear port-{H}amiltonian partial differential equations on networks}.
\newblock {\em PAMM}, 19(1):e201900399, 2019.

\bibitem{Lub08}
\textsc{C.~Lubich}.
\newblock \href{https://doi.org/10.4171/067}{{\em From Quantum to Classical Molecular Dynamics: {R}educed Models and Numerical Analysis}}.
\newblock Zurich lectures in advanced mathematics. European Mathematical Society, 2008.

\bibitem{Marsden1983}
\textsc{J.~E. Marsden, A.~Weinstein, T.~S. Ratiu, R.~Schmid, and R.~Spencer}.
\newblock Hamiltonian systems with symmetry, coadjoint orbits and plasma physics.
\newblock In {\em Proceedings of the IUTAM-ISIMM symposium on modern developments in analytical mechanics}, volume 117, pages 289--340, 1983.

\bibitem{mehrmann2023control}
\textsc{V.~Mehrmann and B.~Unger}.
\newblock Control of port-{H}amiltonian differential-algebraic systems and applications.
\newblock {\em Acta Numerica}, 32:395--515, 2023.

\bibitem{polyuga2012effort}
\textsc{R.~V. Polyuga and A.~J. {van der Schaft}}.
\newblock Effort-and flow-constraint reduction methods for structure preserving model reduction of port-{H}amiltonian systems.
\newblock {\em Systems \& Control Letters}, 61(3):412--421, 2012.

\bibitem{Rettberg2023}
\textsc{J.~Rettberg, D.~Wittwar, P.~Buchfink, A.~Brauchler, P.~Ziegler, J.~Fehr, and B.~Haasdonk}.
\newblock \href{https://doi.org/10.1080/13873954.2023.2173238}{Port-{H}amiltonian fluid-structure interaction modelling and structure-preserving model order reduction of a classical guitar}.
\newblock {\em Mathematical and Computer Modelling of Dynamical Systems}, 29(1):116--148, 2023.

\bibitem{SCHWERDTNER2023105655}
\textsc{P.~Schwerdtner, T.~Moser, V.~Mehrmann, and M.~Voigt}.
\newblock \href{https://doi.org/https://doi.org/10.1016/j.sysconle.2023.105655}{Optimization-based model order reduction of port-{H}amiltonian descriptor systems}.
\newblock {\em Systems \& Control Letters}, 182:105655, 2023.

\bibitem{vanderSchaft2014}
\textsc{A.~J. {van der Schaft} and D.~Jeltsema}.
\newblock \href{https://doi.org/10.1561/2600000002}{Port-{H}amiltonian systems theory: An introductory overview}.
\newblock {\em Foundations and Trends® in Systems and Control}, 1(2-3):173--378, 2014.

\end{thebibliography}


\end{document}